\documentclass{amsart}
\usepackage{}
\usepackage{amsfonts}
\usepackage{amsfonts,amssymb,amsmath,amsthm}
\usepackage{mathrsfs}
\usepackage{graphicx}
\usepackage{url}
\usepackage{enumerate}

\newtheorem{theorem}{Theorem}[section]
\newtheorem{lemma}[theorem]{Lemma}
\newtheorem{proposition}[theorem]{Proposition}

\theoremstyle{definition}

\newtheorem{remark}{Remark}
\numberwithin{equation}{section}

\DeclareMathOperator{\vol}{vol}\DeclareMathOperator{\id}{id}

\DeclareMathOperator{\Pf}{Pf}\DeclareMathOperator{\V}{V}

\author{Wei Zhao}
\address{
Department of Mathematics\\
East China University of Science and Technology\\
Shanghai, China}
\email{szhao\underline{ }wei@yahoo.com}

\keywords{Gauss-Bonnet-Chern formula, Finsler manifold, index theory, Thom form}
\subjclass[2010]{Primary 53B40, Secondary 47A53}
\begin{document}

\title[]{A simple proof of the Gauss-Bonnet-Chern formula for Finsler manifolds}

\begin{abstract}From the point of view of index theory, we give a simple proof of a Gauss-Bonnet-Chern formula for all Finsler manifolds by the Cartan connection. Based on this, we establish a Gauss-Bonnet-Chern formula for any metric-compatible connection.
\end{abstract}
\maketitle

\section{Introduction}
In \cite{Ch1,Ch2}, S. S. Chern gave a intrinsic proof of the Gauss-Bonnet-Chern (GBC) formula for all oriented closed $n$-dimensional Riemannian manifolds $(M,g)$
\[
-\int_M\mathbf{\Omega}=\chi(M),\eqno(1.1)
\]
where
\begin{align*}
\mathbf{\Omega}&=\left\{
\begin{array}{lll}
&\frac{(-1)^{p-1}}{2^{2p}\pi^pp!}\epsilon_{i_1\ldots i_{2p}}\Omega_{i_1}^{i_2}\wedge\cdots\wedge \Omega_{i_{2p-1}}^{i_{2p}},&n=2p,\\
\\
&0,&n=2p+1.
\end{array}
\right.
\end{align*}
and $(\Omega_i^j)$ is the local curvature form of the Levi-Civita connection on $M$. Let $\pi:SM\rightarrow M$ be the sphere bundle of $M$ and $S_xM$ be the fibre on $\pi^{-1}(x)$. Chern's idea is to show that $\pi^*\mathbf{\Omega}$ is the derivative of a "total curvature" $\Pi$ on $SM$ satisfying $\Pi|_{S_xM}$ is the normalized volume form of the fibre. To realize this, Chern constructed two polynomials $\Phi_k$, $\Psi_k$ miraculously and found the important relationship
\[
d\Phi_k=\Psi_{k-1}+\frac{n-2k-1}{2(k+1)}\Psi_k.\eqno(1.2)
\]
By (1.2), he succeed in constructing the total curvature
\begin{align*}
\Pi=\left\{
\begin{array}{lll}
&\frac{1}{(2\pi)^p}\overset{p-1}{\underset{k=0}{\sum}}\frac{(-1)^k}{(2p-2k-1)!!k!2^k}\Phi_k,&n=2p,\\
\\
&\frac{1}{\pi^p2^{2p+1}p!}\overset{p}{\underset{k=0}{\sum}}(-1)^{k+1}\left(
\begin{array}{lll}
p\\
k
\end{array}
\right)\Phi_k,&n=2p+1.
\end{array}
\right.
\end{align*}
such that
\[
\pi^*\mathbf{\Omega}=d\Pi.\eqno(1.3)
\]
Thus,
pulling back (1.3) by a unit vector with isolated zeros and using the Poincar\'e-Hopf theorem, one can obtain the formula (1.1).

Let $(M,F)$ be a $n$-dimensional Finsler manifold. Denote by $\pi:SM\rightarrow M$ the projective sphere bundle and $\pi^*TM$ the pull-back bundle. $F$ induces a natural Riemannian metric on $\pi^*TM$. There are many important linear connection on $\pi^*TM$, but none of them is both "torsion-free" and "metric-compatible". For example, the Cartan connection is metric-compatible while the Chern connection is torsion-free. Refer to \cite{AIM,BCS} for other interesting connections.

It is natural to ask whether an analogue of (1.1) still holds for Finsler manifolds. And it begins with a work of Lichnerowicz's\cite{Li}. In that paper, Lichnerowicz obtained a GBC formula for Cartan-Berwald spaces by the Cartan connection. One interesting feature of a $n$-dimensional Cartan-Berwald manifold is $\V(x)=\vol(\mathbb{S}^{n-1})$, where $\V(x)$ is the Riemannian volume of $S_xM$ induced by $F$ (see \cite{BC,BS} or Section 2 below). Fifty years later, Bao and Chern\cite{BC} reconsidered this problem and established a GBC formula for Finselr manifolds with $\V(x)=\text{constant}$ by using the Chern connection. In the same year, Shen\cite{Sh1} obtained several formulas of GBC type by the Cartan connection for a certain class of Finsler manifolds. Recently, Lackey\cite{L} used a nice trick to deal with $\V(x)$ and generalized the result of Bao and Chern\cite{BC} to all Finsler manifolds. In fact, Lackey established a GBC formula for any torsion-free connection. The same technique also appeared in an unpublished work of Shen\cite{Sh2} and a GBC formula for any metric-compatible connection was established.

All the methods used in References\cite{Li,BC,Sh1,L,Sh2,L} are inspired by Chern's original idea presented before. It should be noticeable that $\Phi_k$ and $\Psi_k$ constructed by the Chern connection don't satisfy (1.2) while $\Pi$ constructed by the Cartan connection is not a "total curvature". Hence, whichever connection is chosen, one needs a lot of techniques to calculate and correct (1.2) and (1.3), not to mention an arbitrary torsion-free or metric-compatible connection.

Note that the GBC formula (1.1) is the simplest case of the Atiyah-Singer index theorem\cite{BGV,MQ}. The purpose of this paper is to give a simple proof of a GBC formula for any metric-compatible connection and for all Finsler manifold from the point of view of index theory.
Given any metric-compatible connection $D$ on $\pi^*TM$, we define
\begin{align*}
\mathbf{\Omega}^D&=\left\{
\begin{array}{lll}
&\frac{(-1)^{p-1}}{2^{2p}\pi^pp!}\epsilon_{i_1\ldots i_{2p}}\Omega_{i_1}^{i_2}\wedge\cdots\wedge \Omega_{i_{2p-1}}^{i_{2p}},&n=2p,\\
\\
&0,&n=2p+1.
\end{array}
\right.
\end{align*}
where $(\Omega^i_j)$ is the curvature of $D$. First, for the Cartan connection $\nabla$, we have
\begin{theorem}
Let $(M,F)$ be a closed Finsler $n$-manifold and $X$ be a vector field with isolated zeros $\{x_i\}$. Then we have
\[
-\int_{M}[X]^*\left(\frac{\mathbf{\Omega}^\nabla+\mathfrak{D}}{\V(x)}\right)=\frac{\chi(M)}{\vol(\mathbb{S}^{n-1})},
\]
where $[X]: M\backslash\cup\{x_i\}\rightarrow SM$ is the section induced by $X$, $\V(x)$ is the Riemannian volume of $S_xM$, and $\mathfrak{D}$ is some $n$-form on $SM$.
\end{theorem}

It is noticeable that $(1-s)D+s\nabla$ is a metric-compatible connection, for any metric-compatible connection $D$.
Thus, by Theorem 1.1 and a transgression formula (see Section 4), we obtain the following
\begin{theorem}Let $(M,F)$ be a closed Finsler $n$-manifold and $X$ be a vector field with isolated zeros.
Given any metric-compatible connection $D$, we have
\[
-\int_{M}[X]^*\left(\frac{\mathbf{\Omega}^D+\mathfrak{E}}{\V(x)}\right)=\frac{\chi(M)}{\vol(\mathbb{S}^{n-1})},
\]
where $(\mathfrak{E}-\mathfrak{D})$ is some exact $n$-form on $SM$.
\end{theorem}
See Section 4 below for the precise formulas of $\mathfrak{D}$ and $\mathfrak{E}$. It is remarkable that Theorem 1.1 and Theorem 1.2 are independent of the choice of the vector field $X$.

In the Riemannian case, $\mathbf{\Omega}^\nabla=\pi^*\mathbf{\Omega}$, $\V(x)=\vol(\mathbb{S}^{n-1})$, $\int_M[X]^*\mathfrak{D}=0$ and $(\mathfrak{E}-\mathfrak{D})$ is an exact $n$-form pulled back from $M$ (see Remark 3-4 below). Since $[X]^*\pi^*=\id$, both Theorem 1.1 and Theorem 1.2 imply the GBC formula for Riemannian manifolds\cite{BGV,Ch1,Ch2}.

We remark that Shen\cite{Sh2} also established a GBC formula for any metric-compatible connection and for all Finsler manifolds. But the formula and the method in \cite{Sh2} are different from ours here. Refer to \cite{Sh2} for more details.

\proof[Acknowledgements]The authors wish to thank Professor Y-B. Shen for his advice and encouragement.
This work was supported partially
by National Natural Science Foundation of China (Grant No.
11171297).

\section{Preliminaries}
In this paper, the rules that govern our index gymnastics are as follows: Latin indices run from $1$ to $n$; Greek indices run from $1$ to $n-1$.

A Finsler $n$-manifold $(M,F)$ is an $n$-dimensional differential manifold $M$ equipped with a Finsler metric $F$ which is a nonnegative function on $TM$ satisfying the following two conditions:

(1) $F$ is positively homogeneous, i.e., $F(\lambda y)=\lambda F(y)$, for any $\lambda>0$ and $y\in TM$;

(2) $F$ is smooth on $TM\backslash\{0\}$ and the Hessian $\frac{1}{2}[F^2]_{y^iy^j}(x,y)$ is positive definite, where $F(x,y):=F(y^i\frac{\partial}{\partial x^i}|_x)$.

Let $\pi:SM\rightarrow M$ and $\pi^*TM$ be the projective sphere bundle and the pullback bundle, respectively. For each $(x,[y])\in SM$, the distinguished section $\ell$ of $\pi^*TM$ is defined by
\[
\ell_{(x,[y])}=\frac{y^i}{F(y)}{\partial_i},
\]
where $\partial_i:=(x,[y],\frac{\partial}{\partial x^i}|_x)$, $i=1,\cdots, n$ denote the local natural frame of $\pi^*TM$.
The Finsler metric $F$ induces a natural Riemannian metric $g:=g_{ij}\,dx^i\otimes dx^j$ and the Cartan tensor $A:=A_{ijk}\,dx^i\otimes dx^j\otimes dx^k$ on  $\pi^*TM$, where
\[
g_{ij}(x,[y]):=\frac12\frac{\partial^2
F^2(x,y)}{\partial y^i\partial
y^j},\ \ A_{ijk}(x,[y]):=\frac{F}{4}\frac{\partial^3 F^2(x,y)}{\partial
y^i\partial y^j\partial y^k}.
\]

In general, there is no linear connection in $\pi^*TM$ such that it is not only "torsion-free" but also "metric-compatible". There are two important connections  in $\pi^*TM$, which are the Cartan connection\cite{C} and the Chern connection\cite{BCS}. The former is compatible with $g$ while the latter is torsion-free. From now on, we denote by $\nabla$ the Cartan connection and $\bar{\nabla}$ the Chern connection.

Throughout this paper, we assume that $\{e_i\}$ is a local orthonormal frame field for $\pi^*TM$, where $e_n=\ell$, and $\{\omega^i\}$ is the dual frame field. Let $\nabla e_i=:\varpi_i^j\otimes e_j$ and $\bar{\nabla} e_i=:\bar{\varpi}_i^j\otimes e_j$. According to \cite{BCS,C}, these two connections are characterized by the following structure
equations:
\begin{align*}
\text{Cartan connection }&\nabla:\left\{
\begin{array}{lll}
&d\omega^i-\omega^j\wedge \varpi^i_j=-A^i_{j\alpha}\omega^j\wedge \varpi^\alpha_n,\\
&\varpi^i_j+\varpi^j_i=0.
\end{array}
\right.\\
\text{Chern connection }&\bar{\nabla}:\left\{
\begin{array}{lll}
&d\omega^i-\omega^j\wedge \bar{\varpi}^i_j=0,\\
&\bar{\varpi}^i_j+\bar{\varpi}^j_i=-2A_{i\alpha}^j\bar{\varpi}^\alpha_n.
\end{array}
\right.
\end{align*}
In fact, $\varpi^j_i=\bar{\varpi}^j_i+A^j_{i\alpha}\bar{\varpi}_n^\alpha$ (see \cite[p.39]{BCS}). Since $A^i_{nj}=0$, $\varpi^i_n=\bar{\varpi}^i_n$. Given any $x\in M$, let $S_xM:=\pi^{-1}(x)$ and $i_x:S_xM\hookrightarrow SM$ be the injective map. It follows from \cite{BS} that the Riemannian volume form $d\V(x)$ of $S_xM$ induced by $g$ satisfies
\[
d\V(x)=i^*_x(\varpi^n_1\wedge\cdots\wedge\varpi_{n-1}^n)=i^*_x(\bar{\varpi}_1^n\wedge\cdots\wedge\bar{\varpi}_{n-1}^n).
\]

\section{A transgression formula for the Cartan connection} \label{ns}
For convenience, let $\mathscr{A}^{i,j}:=\Gamma(SM,\wedge^iT^*SM\otimes \wedge^j \pi^*TM)$ and $\mathscr{A}:=\sum_{i,j}\mathscr{A}^{i,j}$. In this section, we will investigate $\mathscr{A}$ and derive a transgression formula for the Cartan connection.

Clearly, $\mathscr{A}$ is a bigraded algebra (cf. \cite{MQ}). Hence, for $a\otimes b\in \mathscr{A}^{i,j}$ and $c\otimes d\in \mathscr{A}^{k,l}$, the product of $a\otimes b$ and $c\otimes d$ is defined by
\[
(a\otimes b)\cdot(c\otimes d)=(-1)^{jk}(a\wedge c)\otimes (b\wedge d).
\]
For each $s\in \Gamma(SM,\pi^*TM)$, the contraction $\iota(s):\mathscr{A}^{i,j}\rightarrow \mathscr{A}^{i,j-1}$ is defined by
\[
\iota (s)(\alpha\otimes s_1\wedge\ldots\wedge s_j)=\underset{k}{\sum}(-1)^{i+k-1}g(s,s_k)\,\alpha \otimes (s_1\wedge\ldots\wedge\hat{s_k}\wedge\ldots \wedge s_j).
\]

We always identify $\mathfrak{so}(\pi^*TM)$ with $\wedge^2 \pi^*TM$ by the map
\[
B\in \mathfrak{so}(\pi^*TM)\rightarrow \frac12\underset{i,j}{\sum}g(Be_i,e_j)e_i\wedge e_j.
\]
Note that the Cartan connection is a operator from $\mathscr{A}^{i,j}$ to $\mathscr{A}^{i+1,j}$, i.e.,
\[
\nabla(a\otimes b)=da\otimes b+(-1)^ia\wedge \nabla b , \ \forall\, a\otimes b \in \mathscr{A}^{i,j}.
\]
In particular, $\nabla\Omega=0$. Here, we view the curvature of the Cartan connection $\Omega$ as an element of $\mathscr{A}^{2,2}$.

An easy calculation yields the following proposition, which is useful in this paper.
\begin{proposition}
For each $s\in \mathscr{A}^{0,1}$, $\alpha\in \mathscr{A}^{i,j}$ and $\beta\in \mathscr{A}^{k,k}$, we have
\begin{align*}
(1)&\ \alpha\cdot\beta=\beta\cdot\alpha;\\
(2)&\ \iota (s)(\alpha\cdot\beta)=(\iota (s)\alpha)\cdot \beta+(-1)^{i+j}\alpha\cdot (\iota (s)\beta);\\
(3)&\ \nabla(\alpha\cdot\beta)=(\nabla\alpha)\cdot\beta+(-1)^{i+j}\alpha\cdot(\nabla\beta).
\end{align*}
\end{proposition}

\begin{remark}
In fact, Proposition 3.1 holds for any connection on $\pi^*TM$.
\end{remark}

It is not hard to see that
\[
0=\iota(\ell)\nabla\ell,\ \nabla(\nabla \ell)=\iota(\ell)\Omega,\  \iota(\ell)\ell=1. \eqno(3.1)
\]
(3.1) together with Proposition 3.1 yields immediately the following proposition.
\begin{proposition}Define $\Theta_t:=t^2/2+it\nabla \ell+\Omega$. Then
\[
(\nabla-it \iota(\ell))f(\Theta_t)=0,
\]
where $f:\mathbb{R}\rightarrow \mathbb{R}$ is a smooth function and
\[
f(\Theta_t):=\overset{\infty}{\underset{k=0}{\sum}}\frac{f^{(k)}(t^2/2)}{k!}(it\nabla \ell+\Omega)^k.
\]
\end{proposition}

Let $\mathscr{A}(SM):=\sum_{i}\mathscr{A}^{i,0}$.
Since $\pi^*TM$ is an oriented bundle, we can induce the Berezin integral $\mathscr{B}$ (cf. \cite{BGV}) to $\mathscr{A}$ such that
\[
\mathscr{B}:a\otimes \eta\in \mathscr{A}\mapsto a(\mathscr{B}\eta)\in \mathscr{A}(SM),
\]
with
\[
\mathscr{B}(e_I)=\left\{
\begin{array}{lll}
&\epsilon_I,&|I|=n,\\
&0,&\text{otherwise}.
\end{array}
\right.
\]
Moreover, it follows from \cite[Proposition 1.50]{BGV} that
\[
d \mathscr{B}(\xi)=\mathscr{B}(\nabla\xi), \ \text{ for any }\xi\in \mathscr{A}.\eqno(3.2)
\]

Combining Proposition 3.2 and (3.2), we have the following
\begin{lemma}
$U_t:=\mathscr{B}({e^{-\Theta_t}})$ is a closed $n$-form on $SM$.
\end{lemma}
\begin{proof}
Since ${e^{-\Theta_t}}\in \mathscr{A}^{n,n}$, $U_t$ is a $n$-form on $SM$. Now it follows from (3.2) and Proposition 3.2 that
\[
dU_t=\mathscr{B}[(\nabla-it\iota(\ell))({e^{-\Theta_t}})]=0.
\]
\end{proof}
\begin{remark}Mathai-Quillen's proof of the formula (1.1) was carried out by constructing a Thom form on $TM$, which pulled back by the zero-section is exactly the Euler form. Although $U_1$ is similar to the Mathai-Quillen's Thom form restricted to $SM$, the argument in \cite{MQ} cannot be applied in Finsler manifolds directly, since the connection form of any metric-compatible (or torsion-free) connection on $\pi^*TM$ cannot be extended to the zero-section for a general Finsler metric.
\end{remark}
From above, we obtain the following transgression formula.
\begin{lemma}
\[
\frac{d}{dt}U_t=-i d\left[\mathscr{B}\left(\ell\cdot {e^{-\Theta_t}}\right)\right].
\]
\end{lemma}
\begin{proof}In view of (3.1), we have
$\frac{d}{dt}\Theta_t=i(\nabla-it\iota(\ell))\ell$. This together with Lemma 3.3 and (3.2) yields
\begin{align*}
\frac{d}{dt}U_t=&-\mathscr{B}\left((i(\nabla-it\iota(\ell))\ell)\cdot{e^{-\Theta_t}}\right)\\
=&-i \mathscr{B}((\nabla-it\iota(\ell)) (\ell\cdot{e^{-\Theta_t}}))\\
=&-i d[ \mathscr{B}(\ell\cdot{e^{-\Theta_t}})].
\end{align*}
\end{proof}

\section{Proofs of Theorem 1.1 and Corollary 1.2}
Recall that $\Omega\in \mathscr{A}^{2,2}$. According to \cite[Definition 1.35]{BGV}, the Pfaffian of $-\Omega$ is defined by
\[
\Pf(-\Omega):=\mathscr{B}(\exp(-\Omega)).\eqno(4.1)
\]
Note that  $U_t=e^{-t^2/2}\mathscr{B}(e^{-(it\nabla\ell+\Omega)})\rightarrow 0$ (as $t\rightarrow\infty$). Hence, Lemma 3.4 yields
\[
\Pf(-\Omega)=U_0=i d\left[\int^\infty_0\mathscr{B}(\ell\cdot{e^{-\Theta_t}})dt\right]\ \in \mathscr{A}^n(SM).
\]
Denote by $\Xi$ the component of $e^{-(ti\nabla\ell+\Omega)}$ in $\mathscr{A}^{n-1,n-1}$. Thus,
\[
\Pf(-\Omega)=id \left[\int^\infty_0 e^{-t^2}\mathscr{B}(\ell\cdot\Xi) dt\right].\eqno(4.2)
\]
Using this simple observation, we easily get the following formula.
\begin{lemma}
\[
\Pf(-\Omega)=(-1)^{n-1}d\left(\overset{\left[\frac{n-1}{2}\right]}{\underset{k=0}{\sum}}\frac{(-1)^k2^{\frac{n}{2}}\Phi_k}{k!(n-1-2k)!2^{2k+1}}\Gamma\left(\frac{n-2k}{2}\right)\right),
\]
where $\Gamma(s)$ is the gamma function and
\[
\Phi_k:=\sum\epsilon_{\alpha_1\ldots \alpha_{n-1}}\Omega_{\alpha_1}^{\alpha_2}\wedge\cdots\wedge\Omega_{\alpha_{2k-1}}^{\alpha_{2k}}\wedge \varpi_{\alpha_{2k+1}}^n\wedge\cdots\wedge \varpi_{\alpha_{n-1}}^n.
\]
\end{lemma}
\begin{proof}
It is not hard to see that
\begin{align*}
\Xi&=\overset{\infty}{\underset{k=0}{\sum}}\frac{(-1)^k}{k!}\left(\underset{\{s:(k-s)+2s=n-1,\,0\leq s \leq k\}}{\sum}
\left(
\begin{array}{lll}
k\\
s
\end{array}
\right)
(ti\nabla\ell)^{k-s}\cdot\Omega^s\right)\\
&=\overset{n-1}{\underset{k=\left[\frac{n}{2}\right]}{\sum}}\frac{(-1)^k}{(n-1-k)!(2k-(n-1))!}(ti\nabla\ell)^{2k-(n-1)}\cdot\Omega^{n-1-k}\\
&=(-i)^{n-1}\overset{\left[\frac{n-1}{2}\right]}{\underset{k=0}{\sum}}\frac{(t\nabla\ell)^{n-1-2k}\cdot\Omega^{k}}{k!(n-1-2k)!}.
\end{align*}

Clearly,
\begin{align*}
(\nabla\ell)^k&=(-1)^{\frac{k(k-1)}{2}}\varpi^{\alpha_1}_n\wedge\cdots\wedge\varpi^{\alpha_k}_n\otimes e_{\alpha_1}\wedge\cdots\wedge e_{\alpha_k},\\
\Omega^k&=\frac{1}{2^k}\Omega_{j_1}^{j_2}\wedge\cdots\wedge \Omega_{j_{2k-1}}^{j_{2k}}\otimes e_{j_1}\wedge\cdots\wedge e_{j_{2k}}.
\end{align*}
Combining $\ell=e_n$ and all the equalities above, we have
\[
\ell\cdot\Xi=\frac{(-1)^{n-1}}{\epsilon(n+1)}\overset{\left[\frac{n-1}{2}\right]}{\underset{k=0}{\sum}}\frac{t^{n-1-2k}(-1)^k}{k!(n-1-2k)!2^k}\Phi_k\otimes e_1\wedge\cdots\wedge e_n,
\]
where \begin{align*}
\epsilon(n):=\left\{
\begin{array}{lll}
&1,\ n=2p,\\
&i,\ n=2p+1.
\end{array}
\right.
\end{align*}
Now it follows from (4.1) that
\begin{align*}
\Pf(-\Omega)
=(-1)^{n-1}\epsilon(n)d\left(\overset{\left[\frac{n-1}{2}\right]}{\underset{k=0}{\sum}}\frac{(-1)^k2^{\frac{n}{2}}\Phi_k}{k!(n-1-2k)!2^{2k+1}}\Gamma\left(\frac{n-2k}{2}\right)\right).
\end{align*}
However, if $n=2p+1$, then (4.1) implies $\Pf(-\Omega)=0$ and (therefore)
\[
(-1)^{n-1}d\left(\overset{\left[\frac{n-1}{2}\right]}{\underset{k=0}{\sum}}\frac{(-1)^k2^{\frac{n}{2}}\Phi_k}{k!(n-1-2k)!2^{2k+1}}\Gamma\left(\frac{n-2k}{2}\right)\right)
=0=\Pf(-\Omega).
\]
\end{proof}

Define $\mathbf{\Omega}^\nabla:=\frac{-1}{(2\pi)^{\frac{n}{2}}}\Pf(-\Omega)$ and
\[
\Pi:=\frac{(-1)^n}{\pi^{\frac{n}{2}}}\left(\overset{\left[\frac{n-1}{2}\right]}{\underset{k=0}{\sum}}\frac{(-1)^k\Phi_k}{k!(n-1-2k)!2^{2k+1}}\Gamma\left(\frac{n-2k}{2}\right)\right).
\]
An easy calculation yields that
\begin{align*}
\mathbf{\Omega}^\nabla&=\left\{
\begin{array}{lll}
&\frac{(-1)^{p-1}}{2^{2p}\pi^pp!}\epsilon_{i_1\ldots i_{2p}}\Omega_{i_1}^{i_2}\wedge\cdots\wedge \Omega_{i_{2p-1}}^{i_{2p}},&n=2p,\\
\\
&0,&n=2p+1.
\end{array}
\right.\\
\Pi&=\left\{
\begin{array}{lll}
&\frac{1}{(2\pi)^p}\overset{p-1}{\underset{k=0}{\sum}}\frac{(-1)^k}{(2p-2k-1)!!k!2^k}\Phi_k,&n=2p,\\
\\
&\frac{1}{\pi^p2^{2p+1}p!}\overset{p}{\underset{k=0}{\sum}}(-1)^{k+1}\left(
\begin{array}{lll}
p\\
k
\end{array}
\right)\Phi_k,&n=2p+1.
\end{array}
\right.
\end{align*}
Hence, $\Omega^\nabla$ and $\Pi$ are of the same form as the ones defined in \cite{Ch1,Ch2,BC,Sh1}.
And Lemma 4.1 implies that $\mathbf{\Omega}^\nabla=d\Pi$. Clearly, this formula holds for any metric-compatible connection. For simplicity, set
\begin{align*}
\Upsilon_1&:=\frac{(-1)^n}{2\pi^{\frac{n}{2}}}\frac{\Phi_0}{(n-1)!}\Gamma\left(\frac{n}{2}\right),\\ \Upsilon_2&:=\frac{(-1)^n}{\pi^{\frac{n}{2}}}\left(\overset{\left[\frac{n-1}{2}\right]}{\underset{k=1}{\sum}}\frac{(-1)^k\Phi_k}{k!(n-1-2k)!2^{2k+1}}\Gamma\left(\frac{n-2k}{2}\right)\right),\\
\mathfrak{D}&:=-{d\Upsilon_2}-d\log\V(x)\wedge \Upsilon_1.
\end{align*}

Applying the same technique as in \cite{Sh1}, we obtain the following lemma.
\begin{lemma}Given any vector filed $X$ on $M$ with isolated zeros $\{x\}$, let $[X]:M\backslash \cup \{x\}\rightarrow SM \backslash \cup S_{x}M$ denote the section induced by $X$. Then, for each isolate zero $x$ and each small $\epsilon>0$,
\[
\int_{\partial B^+_{x}(\epsilon)}[X]^*\left(\frac{\Upsilon_1}{\V(x)}\right)=(-1)^n\frac{\text{index}(X;x)}{\vol(\mathbb{S}^{n-1})}.
\]
\end{lemma}
\begin{proof}
Define a map $\varphi_\epsilon: S_xM\rightarrow SM$ by
\[
\varphi_\epsilon([y])=[X]\circ\kappa([y]),
\]
where $\kappa([y]):=\exp_x\left(\frac{\epsilon y}{F(y)}\right)$. $(\exp_x)_{*0}=\text{id}$ implies that $\kappa$ is a diffeomorphism between $S_xM$ and $\partial B^+_x(\epsilon)$ and (therefore) $\text{deg}(\kappa)=1$.
Hence,
\[
\text{deg}(\varphi_\epsilon)=\text{deg}([X])\circ\text{deg}(\kappa)=\text{deg}([X])=\text{index}(X;x).
\]
Recall that $\vol(\mathbb{S}^{n-1})=2\pi^{n/2}/(\Gamma(n/2))$ and $\Phi_0|_{S_xM}=(n-1)!d\V(x)$. Thus,
\begin{align*}
&\int_{\partial B^+_{x}(\epsilon)}[X]^*\left(\frac{\Upsilon_1}{\V(x)}\right)=\int_{\kappa(S_xM)}[X]^*\left(\frac{\Upsilon_1}{\V(x)}\right)=\int_{S_xM}\kappa^*\circ[X]^*\left(\frac{\Upsilon_1}{\V(x)}\right)\\
=&\int_{S_xM}\varphi_\epsilon^*\left(\frac{\Upsilon_1}{\V(x)}\right)=\text{deg}(\varphi_\epsilon)\int_{S_xM}\frac{\Upsilon_1}{\V(x)}=(-1)^n\frac{\text{index}(X;x)}{\vol(\mathbb{S}^{n-1})}.
\end{align*}
\end{proof}

From above, we now prove Theorem 1.1.
\begin{proof}[ \textbf{Proof of Theorem 1.1}]
It is easy to see that
\[
\frac{\mathbf{\Omega}^\nabla}{\V(x)}=\frac{d\Pi}{\V(x)}=d\left(\frac{\Pi}{\V(x)}\right)-d\left(\frac1{\V(x)}\right)\wedge \Pi=d\left(\frac{\Upsilon_1}{\V(x)}\right)-\frac{\mathfrak{D}}{\V(x)}.\eqno(4.3)
\]
Let $\{x_i\}$ be the isolated zeros of $X$. Choose a small $\varepsilon>0$ such that the forward balls $B^+_{x_i}(\varepsilon)$ are disjoint from each other. (4.3) together with Lemma 4.2 now yields
\begin{align*}
&\int_{M\backslash \cup B^+_{xi}(\epsilon)}[X]^*\left(\frac{\mathbf{\Omega}^\nabla+\mathfrak{D}}{\V(x)}\right)=\int_{M\backslash \cup B^+_{xi}(\epsilon)}[X]^*d\left(\frac{\Upsilon_1}{\V(x)}\right)\\
=&-\underset{i}{\sum}\int_{\partial B^+_{x_i}(\epsilon)}[X]^*\left(\frac{\Upsilon_1}{\V(x)}\right)= \frac{(-1)^{n+1}}{\vol(\mathbb{S}^{n-1})}\underset{i}{\sum}\text{index}(X;x_i)=(-1)^{n+1}\frac{\chi(M)}{\vol(\mathbb{S}^{n-1})}.
\end{align*}
Note that $\chi(M)=0$ when $n=2p+1$. We finish the proof by letting $\epsilon\rightarrow 0^+$.
\end{proof}

\begin{remark}In the Riemannian case, $\V(x)=\vol(\mathbb{S}^{n-1})$ and $\nabla$ is the pull-back connection induced by the Levi-Civita connection. Hence, $\Upsilon_2$ has no pure-$dy$ part and (therefore) $\int_M[X]^*\mathfrak{D}=\lim_{\epsilon\rightarrow0}\int_{M\backslash\cup B^+_{x_i}(\epsilon)}[X]^*\mathfrak{D}=0$. Thus, the theorem above implies the Gauss-Bonet-Chern theorem\cite{Ch1,Ch2}. However, in the non-Riemannian case, $\Upsilon_2$ may has the pure-$dy$ part even when $\V(x)=\text{const}$. Refer to \cite{Sh1} for more interesting results in this case.
\end{remark}

Let ${D_s}$ denote a family of connections compatible with $g$. Set ${D_s}=:d+\omega_s$, where $\omega_s\in \mathscr{A}^1(SM)\otimes \mathfrak{so}(\pi^*TM)$. Thus, $\frac{\partial{D_s}}{\partial s}=\frac12\sum\frac{\partial(\omega_s)^j_i}{\partial s} e_i\wedge e_j\in \mathscr{A}^{1,2}$ and
\[
\frac{\partial}{\partial s}\Omega_s= {D_s}\left(\frac{\partial {D_s}}{\partial s}\right)\in \mathscr{A}^{2,2},
\]
where $\Omega_s$ is the curvature of ${D_s}$. Define $\mathbf{\Omega}^{{D_s}}:=-\frac{\Pf(-\Omega_s)}{(2\pi)^{n/2}}$. $D_s\Omega_s=0$ together with (3.2) now yields
\begin{align*}
\frac{\partial}{\partial s}\mathbf{\Omega}^{{D_s}}&=\frac{-1}{(2\pi)^{\frac{n}{2}}}\frac{\partial}{\partial s}\mathscr{B}(\exp(-\Omega_s))=\frac{1}{(2\pi)^{\frac{n}{2}}}\mathscr{B}\left({D_s}\left(\exp(-\Omega_s)\cdot\frac{\partial{D_s}}{\partial s}\right)\right)\\
&=\frac{1}{(2\pi)^{\frac{n}{2}}}d\mathscr{B}\left(\exp(-\Omega_s)\cdot\frac{\partial{D_s}}{\partial s}\right).
\end{align*}

By the argument above, we obtain Corollary 1.2 directly.
\begin{proof}[\textbf{Proof of Theorem 1.2}]
Set $D_s=s\nabla+(1-s)D$. The transgression formula above implies that
$\mathbf{\Omega}^\nabla=\mathbf{\Omega}^D+d\Upsilon_3$, where
$\Upsilon_3:={(2\pi)^{\frac{-n}{2}}}\int^1_0\mathscr{B}\left(\exp(-\Omega_s)\cdot\frac{\partial{D_s}}{\partial s}\right)ds$. Set $\mathfrak{E}=\mathfrak{D}+d\Upsilon_3$.
Corollary 1.2 now follows from Theorem 1.1.
\end{proof}
\begin{remark}
In the Riemannian case, Theoerm 1.2 implies that, for any metric-compatible connection $D$, $\int_M \mathbf{\Omega}^D=\chi(M)$, that is, $\mathbf{\Omega}^D$ is a Euler form (cf. \cite[Theorem 1.56]{BGV}).
\end{remark}

\end{document}